\documentclass[12pt,leqno,a4paper]{article}

\usepackage[english]{babel}
\usepackage{amsmath}
\usepackage{amsfonts}
\usepackage{amsthm}
\usepackage{amssymb}

\usepackage{xspace}
\usepackage{euscript}
\usepackage{graphicx}
\usepackage{amscd}
\usepackage{tabularx}

\usepackage{enumerate}
 \usepackage{epsfig} 
 \usepackage{graphics} 
\theoremstyle{plain}
\newtheorem{theo}{Theorem}[section]
\newtheorem{prop}[theo]{Proposition}
\theoremstyle{definition}
\newtheorem{definition}[theo]{Definition}
\theoremstyle{remark}
\newtheorem{rem}[theo]{Remark}
\numberwithin{equation}{section}

\newcommand{\R}{\mathbb{R}}
\newcommand{\N}{\mathbb{N}}

\title{Optimal stability in the identification of a
rigid inclusion in an isotropic Kirchhoff-Love plate}
\author{Antonino Morassi\thanks{Dipartimento Politecnico di Ingegneria e Architettura,
Universit\`a degli Studi di Udine, via Cotonificio 114, 33100
Udine, Italy. E-mail: \textsf{antonino.morassi@uniud.it}}, \  Edi
Rosset\thanks{Dipartimento di Matematica e Geoscienze,
Universit\`a degli Studi di Trieste, via Valerio 12/1, 34127
Trieste, Italy. E-mail: \textsf{rossedi@units.it}} \ and Sergio
Vessella\thanks{Dipartimento di Matematica e Informatica ``Ulisse
Dini'', Universit\`a degli Studi di Firenze, Via Morgagni 67/a,
50134 Firenze, Italy. E-mail: \textsf{sergio.vessella@unifi.it}}}

\begin{document}

\maketitle

\begin{abstract}

In this paper we consider the inverse problem of determining a
rigid inclusion inside a thin plate by applying a couple field at
the boundary and by measuring the induced transversal displacement
and its normal derivative at the boundary of the plate. The plate
is made by non-homogeneous, linearly elastic and isotropic
material. Under suitable a priori regularity assumptions on the
boundary of the inclusion, we prove a constructive stability
estimate of log type. Key mathematical tool is a recently proved optimal three
spheres inequality at the boundary for solutions to the
Kirchhoff-Love plate's equation.

\medskip

\noindent\textbf{Mathematics Subject Classification (2010)}:
Primary 35B60. Secondary 35B30, 35Q74, 35R30.

\medskip

\noindent \textbf{Keywords}: Inverse problems, elastic plates,
stability estimates, unique continuation, rigid inclusion.
\end{abstract}

\centerline{}

\section{Introduction}

In this paper we consider the inverse problem of the stable
determination of a rigid inclusion embedded in a thin elastic
plate by measuring the transverse displacement and its normal
derivative at the boundary induced by a couple field applied at
the boundary of the plate. We prove that the stability estimate of
log-log type found in \cite{M-R-V2} can be improved to a single
logarithm in the case in which the plate is made of isotropic
linear elastic material. {}From the point of view of applications,
modern requirements of structural condition assessment demand the
identification of defects using non-destructive methods, and,
therefore, the present results can be useful in quality control of
plates. We refer, among other contributions, to Bonnet and
Constantinescu \cite{B-C} for a general overview of inverse
problems arising in diagnostic analysis applied to linear
elasticity and, in particular, to plate theory (\cite[Section $5.3$]{B-C}),
and to \cite{K} for the identification of a stiff
inclusion in a composite thin plate based on wavelet analysis of
the eigenfunctions.

In order to describe our stability result, let us introduce the
Kirchhoff-Love model of thin, elastic isotropic plate under
infinitesimal deformation; see, for example, \cite{Gu}. Let the
middle plane of the plate $\Omega$ be a bounded domain of $\R^2$
with regular boundary. The rigid inclusion $D$ is modelled as a
simply connected domain compactly contained in $\Omega$. Under the
assumptions of vanishing transversal forces in $\Omega$ and
assigned couple field $\widehat{M}$ acting on $\partial \Omega$,
the transversal displacement $w \in H^2(\Omega)$ of the plate
satisfies the following mixed boundary value problem
\begin{center}
\( {\displaystyle \left\{
\begin{array}{lr}
     {\rm div}({\rm div} (
      {\mathbb P}\nabla^2 w))=0,
      & \mathrm{in}\ \Omega \setminus \overline {D},
        \vspace{0.25em}\\
      ({\mathbb P} \nabla^2 w)n\cdot n=-\widehat M_n, & \mathrm{on}\ \partial \Omega,
          \vspace{0.25em}\\
      {\rm div}({\mathbb P} \nabla^2 w)\cdot n+(({\mathbb P} \nabla^2
      w)n\cdot \tau),_s
      =(\widehat M_\tau),_s, & \mathrm{on}\ \partial \Omega,
        \vspace{0.25em}\\
      w|_{\overline{D}} \in \mathcal{A}, &\mathrm{in}\ \overline{D},
          \vspace{0.25em}\\
        \frac{\partial w^e}{\partial n} = \frac{\partial w^i}{\partial n}, &\mathrm{on}\ \partial {D},
          \vspace{0.25em}\\
\end{array}
\right. } \) \vskip -8.9em
\begin{eqnarray}
& & \label{eq:1.dir-pbm-incl-rig-1}\\
& & \label{eq:1.dir-pbm-incl-rig-2}\\
& & \label{eq:1.dir-pbm-incl-rig-3}\\
& & \label{eq:1.dir-pbm-incl-rig-4}\\
& & \label{eq:1.dir-pbm-incl-rig-5}
\end{eqnarray}
\end{center}
coupled with the \emph{equilibrium conditions} for the rigid
inclusion $D$
\begin{multline}
  \label{eq:1.equil-rigid-incl}
    \int_{\partial D} \left ( {\rm div}({\mathbb P} \nabla^2 w)\cdot n+(({\mathbb P} \nabla^2
  w)n\cdot \tau),_s \right )g - (({\mathbb P} \nabla^2 w)n\cdot n)
  g_{,n} =0, \\   \quad \hbox{for every } g\in \mathcal{A},
\end{multline}
where $\mathcal{A}$ denotes the space of affine functions. We
recall that, {}from the physical point of view, the boundary
conditions
\eqref{eq:1.dir-pbm-incl-rig-4}-\eqref{eq:1.dir-pbm-incl-rig-5}
correspond to ideal connection between the boundary of the rigid
inclusion and the surrounding elastic material, see, for example,
\cite[Section $10.10$]{O-R}. The unit vectors $n$ and $\tau$ are
the outer normal to $\Omega \setminus \overline{D}$ and the
tangent vector to $\partial D$, respectively. Moreover, we have
defined $w^e \equiv w|_{ \Omega \setminus \overline{D}}$ and
$w^i\equiv w|_{\overline{D} }$. The functions $\widehat M_{\tau}$,
$\widehat M_n$ are the twisting and bending component of the
assigned couple field $\widehat M$, respectively. The plate tensor
$\mathbb P$ is given by $\mathbb P = \frac{h^3}{12}\mathbb C$,
where $h$ is the constant thickness of the plate and $\mathbb C$
is the non-homogeneous Lam\'{e} elasticity tensor describing the
response of the material.

The existence of a solution $w \in H^2(\Omega)$ of the problem
\eqref{eq:1.dir-pbm-incl-rig-1}--\eqref{eq:1.equil-rigid-incl} is
ensured by general results, provided that $\widehat{M}\in
H^{-\frac{1}{2}}(\partial\Omega, \R^2)$, with $\int_{\partial
\Omega} \widehat{M}_i = 0$, for $i=1,2$, and $\mathbb P$ is
bounded and strongly convex. Let us notice that $w$ is uniquely
determined up to addition of an affine function.

Let us denote by $w_i$ a solution to
\eqref{eq:1.dir-pbm-incl-rig-1}--\eqref{eq:1.equil-rigid-incl} for
$D=D_i$, $i=1,2$. In order to deal with the stability issue, we
found it convenient to replace each solution $w_i$ with
$v_i=w_i-g_i$, where $g_i$ is the affine function which coincides
with $w_i$ on $\partial D_i$, $i=1,2$. By this approach,
maintaining the same letter to denote the solution, the
equilibrium problem
\eqref{eq:1.dir-pbm-incl-rig-1}--\eqref{eq:1.dir-pbm-incl-rig-5} can
be rephrased in terms of the following mixed boundary value
problem with homogeneous Dirichlet conditions on the boundary of
the rigid inclusion
\begin{center}
\( {\displaystyle \left\{
\begin{array}{lr}
     {\rm div}({\rm div} (
      {\mathbb P}\nabla^2 w))=0,
      & \mathrm{in}\ \Omega \setminus \overline {D},
        \vspace{0.25em}\\
      ({\mathbb P} \nabla^2 w)n\cdot n=-\widehat M_n, & \mathrm{on}\ \partial \Omega,
          \vspace{0.25em}\\
      {\rm div}({\mathbb P} \nabla^2 w)\cdot n+(({\mathbb P} \nabla^2
      w)n\cdot \tau),_s
      =(\widehat M_\tau),_s, & \mathrm{on}\ \partial \Omega,
        \vspace{0.25em}\\
      w=0, &\mathrm{on}\ \partial D,
          \vspace{0.25em}\\
        \frac{\partial w}{\partial n} = 0, &\mathrm{on}\ \partial {D},
          \vspace{0.25em}\\
\end{array}
\right. } \) \vskip -8.9em
\begin{eqnarray}
& & \label{eq:1.dir-pbm-incl-rig-1bis}\\
& & \label{eq:1.dir-pbm-incl-rig-2bis}\\
& & \label{eq:1.dir-pbm-incl-rig-3bis}\\
& & \label{eq:1.dir-pbm-incl-rig-4bis}\\
& & \label{eq:1-dir-pbm-incl-rig-5bis}
\end{eqnarray}
\end{center}
for which there exists a unique solution $w\in H^2(\Omega\setminus
\overline{D})$. The arbitrariness of this normalization, related
to the fact that $g_i$ is unknown, $i=1,2$, leads to the following
formulation of the stability issue.

\textit{Given an open portion $\Sigma$ of $\partial \Omega$,
satisfying suitable regularity assumptions, and given two solutions $w_i$ to
\eqref{eq:1.dir-pbm-incl-rig-1bis}--\eqref{eq:1-dir-pbm-incl-rig-5bis} when $D=D_i$, $i=1,2$, satisfying,
for some $\epsilon>0$,
\begin{equation}
  \label{eq:cond_Dir_stab}
  \min_{g\in \mathcal{A}}\left\{ \| w_1 - w_2 - g \|_{L^2(\Sigma)}
  + \left\| \frac{\partial}{\partial n}(w_1 - w_2 - g) \right\|_{L^2(\Sigma)}\right\}\leq \epsilon,
\end{equation}
to evaluate the rate at which the Hausdorff distance $d_{\cal H}(
\overline{D_1},\overline{D_2} )$ between $D_1$ and $D_2$ tends to
zero as $\epsilon$ tends to zero.}

In this paper we prove the following quantitative stability
estimate of log type for inclusions $D$ of $C^{6,\alpha}$ class:
\begin{equation}
    \label{eq:log-estimate}
    d_{\cal H}( \overline{D_1},\overline{D_2} ) \leq C | \log \epsilon |^{-\eta},
\end{equation}
where $C$, $\eta$, $C>0$ and $\eta>0$, are constants only
depending on the a priori data, see Theorem \ref{theo:Main} for a
precise statement.

The above estimate is an improvement of the log-log type stability
estimate found in \cite{M-R-V2}, although it must be said that the
latter is not restricted to isotropic materials and also applies
to less regular inclusions (e.g., $D$ of $C^{3,1}$ class). It is
worth to notice that a single logarithmic rate of convergence for
the fourth order elliptic equation modelling the deflection of a
Kirchhoff-Love plate is expected to be optimal, as it is in fact
for the analogous inverse problem in the scalar elliptic case,
which models the detection of perfectly conducting inclusions in
an electric conductor in terms of measurements of potential and
current taken on an accessible portion of the boundary of the
body, as shown by the counterexamples due to Alessandrini
(\cite{Al}), Alessandrini and Rondi (\cite{Al-R}), see also \cite{Dc-R}.

The methods used to prove \eqref{eq:log-estimate} are inspired to
the approach presented in the seminal paper \cite{A-B-R-V} where,
for the first time, it was shown how logarithmic stability
estimates for the inverse problem of determining unknown
boundaries can be derived by using quantitative estimates of
Strong Unique Continuation at the Boundary (SUCB) which ensure a
polynomial vanishing rate of the solutions satisfying homogeneous
Dirichlet or Neumann conditions at the boundary. Precisely, in
\cite{A-B-R-V} the key tool was a Doubling Inequality at the
boundary established by Adolfsson and Escauriaza in \cite{l:ae}.

Following the direction traced in \cite{A-B-R-V}, other kinds of
quantitative estimates of the SUCB turned out to be crucial
properties to prove optimal stability estimates for inverse
boundary value problems with unknown boundaries in different
frameworks, see for instance \cite{S} where the case of Robin
boundary condition is investigated. Let us recall, in the context
of the case of thermic conductors involving parabolic equations,
the three cylinders inequality and the one-sphere two-cylinders
inequality at the boundary (\cite{C-R-V1},
\cite{C-R-V2}, \cite{E-F-V}, \cite{E-V}, \cite{Ve1}), and similar
estimate at the boundary for the case of wave equation with time
independent coefficients \cite{S-Ve}, \cite{Ve2}, \cite{Ve3}.

In the present paper, the SUCB property used to improve the double
logarithmic estimate found in \cite{M-R-V2} takes the form of an
optimal three spheres inequality at the boundary. This latter
result was recently proved in \cite{A-R-V} for isotropic elastic
plates under homogeneous Dirichlet boundary conditions, and leads
to a Finite Vanishing Rate at the Boundary (Proposition
\ref{prop:SUCP_boundary}).

Other main mathematical tools are quantitative estimates of Strong
Unique Continuation at the Interior, essentially based on a three
spheres inequality at the interior obtained in \cite{M-R-V1} which
allows to derive quantitative estimates of unique continuation
{}from Cauchy data (Proposition \ref{prop:Cauchy1}), a Lipschitz
estimate of Propagation of Smallness (Proposition \ref{prop:LPS})
and the Finite Vanishing Rate at the Interior (Proposition
\ref{prop:SUCP_interior}) for the solutions to the plate equation.

Let us observe that estimate \eqref{eq:log-estimate} is the first
stability estimate with optimal rate of convergence in the
framework of linear elasticity. Indeed, up to now, the analogous
estimate for the determination, within isotropic elastic bodies,
of rigid inclusions (\cite{Morassi-Rosset2}), cavities
(\cite{Morassi-Rosset1}) or pressurized cavities (\cite{A-B-R})
show a double logarithmic character, and the same convergence rate
has been established by Lin, Nakamura and Wang for star-shaped
cavities inside anisotropic elastic bodies (\cite{L-N-W}).

The plan of the paper is as follows. Main notation and a priori
information are presented in section \ref{Notation}. In section
\ref{Main}, we first state some auxiliary propositions regarding
the estimate of continuation {}from Cauchy data (Proposition
\ref{prop:Cauchy1}) and {}from the interior (Proposition
\ref{prop:LPS}), and the determination of the finite vanishing
rate of the solutions to the plate equation at the interior
(Proposition \ref{prop:SUCP_interior}) and at the Dirichlet
boundary (Proposition \ref{prop:SUCP_boundary}). Finally, in the
second part of section \ref{Main} we give a proof of the main
theorem (Theorem \ref{theo:Main}).

\section{Notation}
\label{Notation}

Let $P=(x_1(P), x_2(P))$ be a point of $\R^2$.
We shall denote by $B_r(P)$ the disk in $\R^2$ of radius $r$ and
center $P$ and by $R_{a,b}(P)$ the rectangle of center $P$ and sides parallel
to the coordinate axes, of length $2a$ and $2b$, namely
$R_{a,b}(P)=\{x=(x_1,x_2)\ |\ |x_1-x_1(P)|<a,\ |x_2-x_2(P)|<b \}$. To simplify the notation,
we shall denote $B_r=B_r(O)$, $R_{a,b}=R_{a,b}(O)$.

Given a bounded domain $\Omega$ in $\R^2$ we shall denote
\begin{equation}
  \label{eq:Omega_rho}
   \Omega_\rho=\{x\in \Omega\ |\ \hbox{dist}(x,\partial\Omega)>\rho\}.
\end{equation}

\noindent When representing locally a boundary as a graph, we use
the following definition.

\begin{definition}
  \label{def:reg_bordo} (${C}^{k,\alpha}$ regularity)
Let $\Omega$ be a bounded domain in ${\R}^{2}$. Given $k,\alpha$,
with $k\in\N$, $0<\alpha\leq 1$, we say that a portion $S$ of
$\partial \Omega$ is of \textit{class ${C}^{k,\alpha}$ with
constants $r_{0}$, $M_{0}>0$}, if, for any $P \in S$, there
exists a rigid transformation of coordinates under which we have
$P=0$ and
\begin{equation*}
  \Omega \cap R_{r_0,2M_0r_0}=\{x \in R_{r_0,2M_0r_0} \quad | \quad
x_{2}>g(x_1)
  \},
\end{equation*}
where $g$ is a ${C}^{k,\alpha}$ function on
$[-r_0,r_0]$
satisfying
\begin{equation*}
g(0)=g'(0)=0,
\end{equation*}
\begin{equation*}
\|g\|_{{C}^{k,\alpha}([-r_0,r_0])} \leq M_0r_0,
\end{equation*}
where
\begin{equation*}
\|g\|_{{C}^{k,\alpha}([-r_0,r_0])} = \sum_{i=0}^k  r_0^i\sup_{[-r_0,r_0]}|g^{(i)}|+r_0^{k+\alpha}|g|_{k,\alpha},
\end{equation*}

\begin{equation*}
|g|_{k,\alpha}= \sup_ {\overset{\scriptstyle t,s\in [-r_0,r_0]}{\scriptstyle
t\neq s}}\left\{\frac{|g^{(k)}(t) - g^{(k)}(s)|}{|t-s|^\alpha}\right\}.
\end{equation*}

\end{definition}

Given a bounded domain $\Omega$ in $\R^2$ such that $\partial
\Omega$ is of class $C^{k,\alpha}$, with $k\geq 1$, we consider as
positive the orientation of the boundary induced by the outer unit
normal $n$ in the following sense. Given a point
$P\in\partial\Omega$, let us denote by $\tau=\tau(P)$ the unit
tangent at the boundary in $P$ obtained by applying to $n$ a
counterclockwise rotation of angle $\frac{\pi}{2}$, that is
\begin{equation}
    \label{eq:2.tangent}
        \tau=e_3 \times n,
\end{equation}
where $\times$ denotes the vector product in $\R^3$ and $\{e_1,
e_2, e_3\}$ is the canonical basis in $\R^3$.

In the sequel we shall denote by $C$ constants which may change
{}from line to line.

\subsection{A priori information}
\medskip
\noindent {\it i) A priori information on the domain.}

Let us consider a thin plate
$\Omega\times[-\frac{h}{2},\frac{h}{2}]$ with middle surface
represented by a bounded domain $\Omega$ in $\R^2$ and having
uniform thickness $h$, $h<<\hbox{diam}(\Omega)$.

We shall assume
that, given $r_0$, $M_1>0$,
\begin{equation}
   \label{eq:bound_area}
\hbox{diam}(\Omega) \leq M_1 r_0.
\end{equation}
We shall also assume that $\Omega$ contains an open simply
connected rigid inclusion $D$ such that
\begin{equation}
   \label{eq:compactness}
    \hbox{dist}(D, \partial \Omega) \geq r_0.
\end{equation}
Moreover, we denote by $\Sigma$ an open portion
within $\partial \Omega$ representing the part of the boundary
where measurements are taken.

Concerning the regularity of the boundaries, given
$M_0\geq\frac{1}{2}$ and $\alpha$, $0<\alpha\leq 1$, we assume that
\begin{equation}
   \label{eq:reg_Omega}
\partial\Omega \hbox{ is of } class\  C^{2,1}
\ with\ constants\ r_0, M_0,
\end{equation}
\begin{equation}
   \label{eq:reg_Sigma}
\Sigma \hbox{ is of } class\  C^{3,1} \ with\ constants\
r_0, M_0.
\end{equation}
\begin{equation}
   \label{eq:reg_D}
\partial D \hbox{ is of } class\  C^{6,\alpha}
\ with\ constants\ r_0, M_0.
\end{equation}

Let us notice that, without loss of generality, we have chosen
$M_0\geq\frac{1}{2}$ to ensure that $B_{r_0}(P) \subset
R_{r_0,2M_0r_0}(P)$ for every $P \in
\partial \Omega$.

Moreover, we shall assume that for some $P_0\in\Sigma$ and some $\delta_0$, $0<\delta_0<1$,
\begin{equation}
   \label{eq:large_enough}
   \partial\Omega\cap
R_{r_0,2M_0r_0}(P_0)\subset\Sigma,
\end{equation}
and that
\begin{equation}
   \label{eq:small_enough}
   |\Sigma|\leq(1-\delta_0)|\partial\Omega|.
\end{equation}
\medskip
\noindent {\it ii) Assumptions about the boundary data.}

On the Neumann data $\widehat{M}$ we assume that
\begin{equation}
   \label{eq:reg_M}
\widehat{M}\in L^2(\partial \Omega,\R^2),\quad
(\widehat{M}_n,(\widehat{M}_\tau),_s)\not\equiv 0,
\end{equation}
\begin{equation}
   \label{eq:supp_M}
\hbox{supp}(\widehat{M})\subset\subset\Sigma,
\end{equation}
the (obvious) compatibility condition
\begin{equation}
   \label{eq:M_comp}
    \int_{\partial \Omega} \widehat{M}_i = 0, \quad i=1,2,
\end{equation}
and that, for a given constant $F>0$,
\begin{equation}
\label{eq:M_frequency}
   \frac{\|\widehat{M}\|_{L^2(\partial
   \Omega ,\R^2)}}{ \|\widehat{M}\|_{H^{-\frac{1}{2}}(\partial \Omega,\R^2)}}\leq
   F.
\end{equation}

{\it iii) Assumptions about the elasticity tensor.}

Let us assume that the plate is made by elastic isotropic material, the plate
tensor $\mathbb P$ is defined by
\begin{equation}
  \label{eq:3.emme_compact}
   \mathbb P A = B \left [ (1-\nu)A^{sym}+\nu (tr A)I_2 \right ],
\end{equation}
for every $2 \times 2$ matrix $A$, where $I_2$ is the $2\times 2$
identity matrix and $\hbox{tr}(A)$ denotes the trace of the matrix
$A$. The
  \emph{bending stiffness} (per unit length) of the plate is given by the
  function
\begin{equation}
  \label{eq:3.stiffness}
  B(x)=\frac{h^3}{12}\left(\frac{E(x)}{1-\nu^2(x)}\right),
\end{equation}
where the \emph{Young's modulus} $E$ and the \emph{Poisson's coefficient} $\nu$
can be written in terms of the Lam\'{e} moduli as follows
\begin{equation}
  \label{eq:3.E_nu}
  E(x)=\frac{\mu(x)(2\mu(x)+3\lambda(x))}{\mu(x)+\lambda(x)},\qquad\nu(x)=\frac{\lambda(x)}{2(\mu(x)+\lambda(x))}.
\end{equation}
Hence, in this case, the displacement equation of equilibrium
\eqref{eq:1.dir-pbm-incl-rig-1} is
\begin{equation}
    \label{eq:equil-eq-isotropic}
    {\rm div}\left ({\rm div} \left ( B((1-\nu)\nabla^2w + \nu \Delta w I_2) \right ) \right )=0, \qquad\hbox{in
    } \Omega.
\end{equation}

We make the following strong convexity assumptions on the Lam\'e moduli
\begin{equation}
  \label{eq:3.Lame_convex}
  \mu(x)\geq \alpha_0>0,\qquad 2\mu(x)+3\lambda(x)\geq\gamma_0>0, \qquad \hbox{a.e. in } \Omega,
\end{equation}
where $\alpha_0$, $\gamma_0$ are positive constants.

We assume that the Lam\'{e} moduli $\lambda,\mu$ satisfy the following regularity assumptions

\begin{equation}
  \label{eq:C4Lame}
  \|\lambda\|_{C^4(\overline{\Omega})}, \quad \|\mu\|_{C^4(\overline{\Omega})}\leq \Lambda_0.
\end{equation}

Under the above assumptions, by standard variational arguments
(see, for example, \cite{l:agmon}), problem
\eqref{eq:1.dir-pbm-incl-rig-1bis}--\eqref{eq:1-dir-pbm-incl-rig-5bis}
has a unique solution $w\in H^2(\Omega\setminus \overline{D})$
satisfying
\begin{equation}
    \label{eq:w_stima_H2}
    \|w\|_{H^2(\Omega\setminus \overline{D})}\leq C r_0^2\|\widehat{M}\|_{H^{-\frac{1}{2}}(\partial\Omega,\R^2)},
\end{equation}
where $C>0$ only depends on $\alpha_0$, $\gamma_0$, $M_0$, and $M_1$.

In the sequel, we shall refer to the set of constants $\alpha_0$,
$\gamma_0$, $\Lambda_0$, $\alpha$, $M_0$, $M_1$, $\delta_0$ and
$F$ as to the \textit{a priori} data.

\section{Statement and proof of the main result}
\label{Main}

Here and in the sequel we shall denote by $G$ the connected
component of $\Omega \setminus   \overline{(D_1 \cup D_2)}$ such
that $\Sigma \subset \partial{G}$.

\begin{theo}[Stability result]
  \label{theo:Main}
Let $\Omega$ be a bounded domain in $\R^2$ satisfying
\eqref{eq:bound_area} and \eqref{eq:reg_Omega}. Let $D_i$,
$i=1,2$, be two simply connected open subsets of $\Omega$
satisfying \eqref{eq:compactness} and \eqref{eq:reg_D}. Moreover,
let $\Sigma$ be an open portion of $\partial\Omega$ satisfying
\eqref{eq:reg_Sigma}, \eqref{eq:large_enough} and
\eqref{eq:small_enough}. Let $\widehat{M}\in
L^2(\partial\Omega,\R^2)$ satisfy
\eqref{eq:reg_M}--\eqref{eq:M_frequency} and let the plate tensor
$\mathbb{P}$ given by \eqref{eq:3.emme_compact} with Lam\'e moduli
satisfying the regularity assumptions \eqref{eq:C4Lame} and the
strong convexity condition \eqref{eq:3.Lame_convex}. Let $w_i\in
H^2(\Omega \setminus \overline{D_i})$ be the solution to
\eqref{eq:1.dir-pbm-incl-rig-1bis}--\eqref{eq:1-dir-pbm-incl-rig-5bis} when $D=D_i$, $i=1,2$. If, given
$\epsilon>0$, we have
\begin{equation}
    \label{eq:small_L2}
\min_{g \in
\cal{A}} \left\{\|w_1 - w_2 -g \|_{L^2(\Sigma)}+r_0
\left\|\frac{\partial}{\partial n}
(w_1 - w_2 -g) \right\|_{L^2(\Sigma)}\right\}\leq
\epsilon,
\end{equation}
then we have

\begin{equation}
    \label{eq:small_Haus_inclusion}
d_{\cal H}( \overline{D_1},\overline{D_2} ) \leq
r_0\omega\left(\frac{\epsilon}
{r_0^2\|\widehat{M}\|_{H^{-\frac{1}{2}}(\partial
\Omega,\R^2)}}\right),
\end{equation}
where $\omega$ is an increasing continuous function on
$[0,\infty)$ which satisfies
\begin{equation}
    \label{eq:omega_log}
\omega(t)\leq C(|\log t|)^{-\eta}, \quad\hbox{for every }t, \
0<t<1,
\end{equation}
and $C$, $\eta$, $C>0$, $\eta>0$, are constants only
depending on the a priori data.
\end{theo}

The proof of Theorem \ref{theo:Main} is obtained {from} a sequence
of propositions. The following proposition can be derived by
merging Proposition $3.4$ of \cite{M-R-V2} and geometrical
arguments contained in Proposition $3.6$ of \cite{A-B-R-V}.


\begin{prop}[Stability Estimate of Continuation
{from} Cauchy Data {\cite[Proposition 3.4]{M-R-V2}}]
  \label{prop:Cauchy1}
Let the hypotheses of Theorem \ref{theo:Main} be satisfied. We
have
\begin{equation}
   \label{eq:Cauchy1}
\int_{(\Omega\setminus \overline{G})\setminus \overline{D_1}}|\nabla^2 w_1|^2\leq
r_0^2\|\widehat{M}\|_{H^{-\frac{1}{2}}(\partial \Omega
,\R^2)}^2\omega\left(\frac{\epsilon}
{r_0^2\|\widehat{M}\|_{H^{-\frac{1}{2}}(\partial
\Omega ,\R^2)}}\right),
\end{equation}
\begin{equation}
   \label{eq:Cauchy1bis}
\int_{(\Omega\setminus \overline{G})\setminus \overline{D_2}}|\nabla^2 w_2|^2\leq
r_0^2\|\widehat{M}\|_{H^{-\frac{1}{2}}(\partial \Omega
,\R^2)}^2\omega\left(\frac{\epsilon}
{r_0^2\|\widehat{M}\|_{H^{-\frac{1}{2}}(\partial
\Omega ,\R^2)}}\right),
\end{equation}
where $\omega$ is an increasing continuous function on
$[0,\infty)$ which satisfies
\begin{equation}
   \label{eq:omega_cauchy_loglog}
\omega(t)\leq C(\log|\log t|)^{-\frac{1}{2}},\qquad \hbox{for every }
t<e^{-1},
\end{equation}
with $C>0$ only depending on
$\alpha_0$, $\gamma_0$, $\Lambda_0$, $M_0$, $M_1$ and $\delta_0$.

Moreover, there exists $d_0>0$, with $\frac{d_0}{r_0}$ only depending on $M_0$,
such that if
$d_{\cal H}( \overline{\Omega\setminus D_1},\overline{\Omega\setminus D_2} )\leq d_0$ then
\eqref{eq:Cauchy1}--\eqref{eq:Cauchy1bis} hold with $\omega$
given by
\begin{equation}
   \label{eq:omega_cauchy_log}
\omega(t)\leq C|\log t|^{-\sigma},\qquad \hbox{for  every }
t<1,
\end{equation}
where $\sigma>0$ and $C>0$ only depend on $\alpha_0$, $\gamma_0$, $\Lambda_0$,
$M_0$, $M_1$, $\delta_0$, $L$ and $\frac{\tilde
r_0}{r_0}$.
\end{prop}

Next two propositions are quantitative versions of the SUCP
property at the interior for solutions to the plate
equilibrium problem. Precisely, Proposition \ref{prop:LPS} has
global character and gives a lower bound of the strain energy
density over any small disc compactly contained in
$\Omega\setminus \overline{D}$ in terms of the Neumann boundary
data. Proposition \ref{prop:SUCP_interior} establishes a
polynomial order of vanishing for solutions to the plate problem
at interior points of $\Omega \setminus \overline{D}$.


\begin{prop}[Lipschitz Propagation of Smallness {\cite[Proposition 3.3]{M-R-V2}}]
  \label{prop:LPS}
Let $\Omega$ be a bounded domain in $\R^2$ satisfying
\eqref{eq:bound_area} and \eqref{eq:reg_Omega}. Let $D$ be an open
simply connected subset of $\Omega$ satisfying
\eqref{eq:compactness}, \eqref{eq:reg_D}. Let $w\in H^2(\Omega
\setminus \overline{D})$ be the solution to
\eqref{eq:1.dir-pbm-incl-rig-1bis}--\eqref{eq:1-dir-pbm-incl-rig-5bis},
coupled with the equilibrium condition
\eqref{eq:1.equil-rigid-incl}, where the plate tensor $\mathbb{P}$
is given by \eqref{eq:3.emme_compact} with Lam\'e moduli
satisfying the regularity assumptions \eqref{eq:C4Lame} and the
strong convexity condition \eqref{eq:3.Lame_convex} and with
$\widehat{M}\in L^2(\partial\Omega,\R^2)$ satisfying
\eqref{eq:reg_M}--\eqref{eq:M_frequency}.

There exists $s>1$, only depending on $\alpha_0$, $\gamma_0$,
$\Lambda_0$, $M_0$ and $\delta_0$, such that for every $\rho>0$
and every $\bar x\in (\Omega\setminus \overline{D})_{s\rho}$, we
have
\begin{equation}
   \label{eq:LPS}
\int_{B_\rho(\bar x)}|\nabla^2 w|^2\geq
\frac{Cr_0^2}{\exp\left[A\left(\frac{r_0}{\rho}\right)^B\right]}
\| \widehat{M} \|_{H^{-\frac{1}{2}}(\partial\Omega,\R^2)}^2,
\end{equation}
where $A>0$, $B>0$ and $C>0$ only depend on $\alpha_0$,
$\gamma_0$, $\Lambda_0$, $M_0$, $M_1$, $\delta_0$ and $F$.
\end{prop}

\begin{rem}
\label{rem:exp_order}
The exponential character of the dependence on $\rho$ in \eqref{eq:LPS} comes {}from the fact that this global estimate follows {}from the trace-type inequality
\begin{equation*}
\| \widehat{M} \|_{H^{-\frac{1}{2}}(\partial\Omega,\R^2)}\leq C\| \nabla ^2 w\|_{L^2(\Omega\setminus \overline{D})},
\end{equation*}
with $C$ only depending on $M_0$, $M_1$, $\delta_0$ and $\Lambda_0$ (see \cite[Lemma 4.6]{M-R-V2}) and {}from a lower bound of the strain energy density over the disc $B_\rho(\overline{x})$ in terms of the strain energy density over all the domain $\Omega\setminus \overline{D}$. The latter estimate requires a geometrical construction involving a number of iterated applications of the three spheres inequality depending in the radius $\rho$ which leads to an exponential dependence. The local polynomial vanishing rate at the interior is given in the following proposition.

\end{rem}


\begin{prop}[Finite Vanishing Rate at the Interior]
\label{prop:SUCP_interior} Under the hypotheses of Proposition
\ref{prop:LPS}, there exist $\overline{c}_0<\frac{1}{2}$ and
$C>1$, only depending on $\alpha_0$, $\gamma_0$, $\Lambda_0$, such
that, for every $\overline{r}>0$ and for every $x \in \Omega
\setminus \overline{D}$ such that $B_{ \overline{r}}(x) \subset
\Omega \setminus \overline{D}$, and for every $r_1<r_2=
\overline{c}_0 \overline{r}$, we have
\begin{equation}
  \label{eq:SUCP_interior}
   \int_{B_{r_{1}}(x)} |\nabla^2 w|^2 \geq
   C\left(\frac{r_1}{\overline{r}}\right)^{\tau_0}
   \int_{B_{\overline{r}}(x)} |\nabla^2 w|^2,
\end{equation}
where $\tau_0\geq 1$ only depends on $\alpha_0$, $\gamma_0$,
$\Lambda_0$, $\delta_0$ and $F$.
\end{prop}

\begin{proof}
The above estimate is based on the following three spheres
inequality at the interior, which was obtained in \cite[Theorem
6.6]{M-R-V1}: there exist $c_0$, $0<c_0<1$, and $C>1$ only
depending on $\alpha_0$, $\gamma_0$, $\Lambda_0$, such that for
every $\overline{r}>0$, for every $x\in (\Omega \setminus
\overline{D})$ such that $B_{\overline{r}}(x) \subset \Omega
\setminus \overline{D}$, and for every $r_1 < r_2 < r_3 <
c_0\overline{r}$, we have
\begin{equation}
  \label{eq:3sph}
   \int_{B_{r_{2}}(x)}|\nabla ^2w|^{2} \leq C
   \left (
   \frac{r_3}{r_1}
   \right )^C
   \left(  \int_{B_{r_{1}}(x)}|\nabla ^2w|^{2}
   \right)^{\vartheta_0}\left(  \int_{B_{r_{3}}(x)}|\nabla ^2w|^{2}
   \right)
   ^{1-\vartheta_0},
\end{equation}
where
\begin{equation}
  \label{eq:3sph-esponente}
   \vartheta_0 =
   \frac
   {\log \left (
   \frac{c_0 r_3}{r_2}
   \right )    }
   { 2  \log \left (
   \frac{r_3}{r_1}
   \right )   }.
\end{equation}
Inequality \eqref{eq:SUCP_interior} can be derived by exploiting
the optimality of the exponent $\vartheta_0$ and by reassembling
the terms in \eqref{eq:3sph}.
\end{proof}

As noticed in the introduction, our key SUCB property is stated in
the following proposition, which is the counterpart at the
boundary $\partial D$ of Proposition \ref{prop:SUCP_interior}.


\begin{prop}[Finite Vanishing Rate at the Boundary]
\label{prop:SUCP_boundary} Under the hypotheses of Proposition
\ref{prop:LPS}, there exist $\overline{c}<\frac{1}{2}$ and $C>1$,
only depending on $\alpha_0$, $\gamma_0$, $\Lambda_0$, $M_0$,
$\alpha$, such that, for every $x\in \partial D$ and for every
$r_1<r_2= \overline{c} r_0$,
\begin{equation}
  \label{eq:SUCP_boundary}
   \int_{B_{r_{1}}(x)\cap (\Omega\setminus \overline{D})}w^2 \geq
   C\left(\frac{r_1}{r_0}\right)^\tau
   \int_{B_{r_{0}}(x)\cap (\Omega\setminus \overline{D})}w^2.
\end{equation}
where $\tau\geq 1$ only depends on $\alpha_0$, $\gamma_0$, $\Lambda_0$,
$M_0$, $\alpha$, $M_1$, $\delta_0$ and $F$.
\end{prop}

\begin{proof}
By Corollary $2.3$ in \cite{A-R-V}, there exist $c<1$, only
depending on $M_0$, $\alpha$, and $C>1$ only depending on
$\alpha_0$, $\gamma_0$, $\Lambda_0$, $M_0$, $\alpha$, such that,
for every $x\in \partial D$ and for every $r_1<r_2<cr_0$,
\begin{equation}
    \label{eq:suc1}
\int_{B_{r_1}(x)\cap(\Omega\setminus \overline{D})}w^2 \geq C\left(\frac{r_1}{r_0}\right)^{\frac{\log B}{\log \frac{cr_0}{r_2}}}
\int_{B_{r_0}(x)\cap(\Omega\setminus \overline{D})}w^2,
\end{equation}
where $B>1$ is given by
\begin{equation}
    \label{eq:suc2}
B= C\left(\frac{r_0}{r_2}\right)^C\frac{\int_{B_{r_0}(x)\cap(\Omega\setminus \overline{D})}w^2}{\int_{B_{r_2}(x)\cap(\Omega\setminus \overline{D})}w^2},
\end{equation}
Let us choose in the above inequalities $r_2 = \overline{c}r_0$, with $\overline{c}=\frac{c}{2}$.

By \eqref{eq:w_stima_H2} we have
\begin{equation}
   \label{eq:suc3}
\int_{B_{r_0}(x)\cap(\Omega\setminus \overline{D})}w^2\leq  C r_0^6\|\widehat{M}\|_{H^{-\frac{1}{2}}(\partial \Omega,\R^2)}^2,
\end{equation}
with $C$ depending on $\alpha_0$, $\gamma_0$, $M_0$, $\alpha$,
$M_1$. By interpolation estimates for solutions to elliptic
equations (see, for instance, \cite[Lemma 4.7]{A-R-V}, stated for
the case of hemidiscs, but which holds also in the present
context), we have that
\begin{equation*}
\int_{B_{r_2}(x)\cap(\Omega\setminus \overline{D})}w^2\geq Cr_2^4
\int_{B_{\frac{r_2}{2}}(x)\cap(\Omega\setminus \overline{D})}|\nabla^2w|^2,
\end{equation*}
with $C$ depending on $\alpha_0$, $\gamma_0$ and $\Lambda_0$.
By Proposition \ref{prop:LPS} and recalling the definition of $r_2$, we derive
\begin{equation}
    \label{eq:suc4}
\int_{B_{r_2}(x)\cap(\Omega\setminus \overline{D})}w^2\geq C r_0^6\|\widehat{M}\|_{H^{-\frac{1}{2}}(\partial \Omega,\R^2)}^2,
\end{equation}
with $C$ depending on $\alpha_0$, $\gamma_0$, $\Lambda_0$, $M_0$, $\alpha$, $M_1$, $\delta_0$ and $F$.

By \eqref{eq:suc3}--\eqref{eq:suc4}, we can estimate $B$ {}from above, obtaining the thesis.
\end{proof}

\begin{proof}[Proof of Theorem \ref{theo:Main}]
In order to estimate the Hausdorff distance between the
inclusions,
\begin{equation}
    \label{eq:1.0}
    \delta = d_{\cal H}( \overline{D_1}, \overline{D_2}),
\end{equation}
it is convenient to introduce the following auxiliary distances:
\begin{equation}
    \label{eq:1.0bis}
   d = d_{\cal H}( \overline{\Omega \setminus D_1}, \overline{\Omega \setminus
    D_2}),
\end{equation}
\begin{equation}
    \label{eq:1.0ter}
    d_m = \max \left \{ \max_{x \in \partial D_1} \hbox{dist}(x, \overline{\Omega \setminus
    D_2}),  \max_{x \in \partial D_2} \hbox{dist}(x, \overline{\Omega \setminus
    D_1}) \right \}.
\end{equation}
Let $\eta >0$ such that
\begin{equation}
    \label{eq:1.1}
    \max_{i=1,2} \int_{ (\Omega \setminus \overline{G}) \setminus
    \overline{D_i}} |\nabla^2 w_i|^2 \leq \eta.
\end{equation}
Following the arguments presented in \cite{A-B-R-V}, we first
control $d_m$ in terms of $\eta$, and then we use this estimate to control $d$ in
terms of $\eta$. Finally, by improving the results in \cite[proof
of Theorem $1.1$, step 2]{C-R-V2}, we estimate $\delta$ in terms of
$d$.

Let us start by proving the inequality
\begin{equation}
    \label{eq:1.2}
    d_m
    \leq
    C r_0
    \left (
    \frac{ \eta   }{ r_0^2 \| \widehat{M}  \|_{H^{-1/2}(\partial \Omega)}^2  }
    \right )^{\frac{1}{\tau}},
\end{equation}
where $\tau$ has been introduced in Proposition \ref{prop:SUCP_boundary} and $C$ is a positive constant only
depending on the a priori data.

Let us assume, without loss of generality, that there exists $x_0
\in \partial D_1$ such that
\begin{equation}
    \label{eq:2.1}
    \hbox{dist}(x_0, \overline{\Omega \setminus D_2}) =d_m>0.
\end{equation}
Since $B_{d_m}(x_0) \subset D_2 \subset \Omega \setminus
\overline{G}$, we have
\begin{equation}
    \label{eq:2.2}
    B_{d_m}(x_0) \cap (\Omega \setminus \overline{D_1}) \subset
    ( \Omega \setminus
\overline{G} ) \setminus \overline{D_1}
\end{equation}
and then, by \eqref{eq:1.1},
\begin{equation}
    \label{eq:2.3}
    \int_{  B_{d_m}(x_0) \cap (\Omega \setminus \overline{D_1}) } |\nabla^2 w_1|^2 \leq \eta.
\end{equation}
Let us assume that
\begin{equation}
    \label{eq:3.1}
    d_m < \overline{c} r_0,
\end{equation}
where $\overline{c}$ is the positive constant appearing in
Proposition \ref{prop:SUCP_boundary}. Since $w_1=0$, $\nabla w_1
=0$ on $\partial D_1$, by Poincar\'{e} inequality (see, for
instance, \cite[Example 4.4]{A-M-R-08})   and noticing that $d_m
\leq \hbox{diam} (\Omega)\leq M_1 r_0$, we have
\begin{equation}
    \label{eq:3.2}
    \eta \geq \frac{C}{r_0^4}  \int_{ B_{d_m}(x_0) \cap (\Omega \setminus \overline{D_1}) } w_1^2
    ,
\end{equation}
where $C>0$ is a positive constant only depending on $\alpha$,
$M_0$, $M_1$.

By Proposition \ref{prop:SUCP_boundary}, we have

\begin{equation}
    \label{eq:4.1}
    \eta \geq \frac{C}{r_0^4} \left(\frac{d_m}{r_0}\right)^\tau \int_{ B_{r_0}(x_0) \cap (\Omega \setminus \overline{D_1}) } w_1^2,
\end{equation}
where $C>0$ is a positive constant only depending on $\alpha_0$, $\gamma_0$, $\Lambda_0$,$\alpha$,
$M_0$, $M_1$ and $F$.

By Lemma 4.7 in \cite{A-R-V}, we have

\begin{equation}
    \label{eq:4.2}
    \eta \geq C\left(\frac{d_m}{r_0}\right)^\tau \int_{ B_{\frac{r_0}{2}}(x_0) \cap (\Omega \setminus \overline{D_1}) } | \nabla^2w_1|^2,
\end{equation}
where $C>0$ is a positive constant only depending on $\alpha_0$, $\gamma_0$, $\Lambda_0$,$\alpha$,
$M_0$, $M_1$.

By Proposition \ref{prop:LPS}, we have

\begin{equation}
    \label{eq:4.3}
    \eta \geq C\left(\frac{d_m}{r_0}\right)^\tau r_0^2\|\widehat{M}\|_{H^{-1/2}(\partial\Omega)}^2,
\end{equation}
where $C>0$ is a positive constant only depending on $\alpha_0$, $\gamma_0$, $\Lambda_0$,$\alpha$,
$M_0$, $M_1$, $\delta_0$, $F$, {}from which we can estimate $d_m$

\begin{equation}
    \label{eq:5.1}
    d_m
    \leq
    C r_0
    \left (
    \frac{ \eta   }{ r_0^2 \| \widehat{M}  \|_{H^{-1/2}(\partial \Omega)}^2  }
    \right )^{\frac{1}{\tau}},
\end{equation}
where $C>0$ is a positive constant only depending on $\alpha_0$, $\gamma_0$, $\Lambda_0$,$\alpha$,
$M_0$, $M_1$, $\delta_0$, $F$.

Now, let us assume that
\begin{equation}
    \label{eq:5.2}
    d_m
    \geq
    \overline{c} r_0.
\end{equation}
By starting again {}from \eqref{eq:2.3}, and applying Proposition \ref{prop:LPS} and recalling $d_m\leq M_1r_0$, we easily have
\begin{equation}
    \label{eq:5.4}
    d_m
    \leq
    C r_0
    \left (
    \frac{ \eta   }{ r_0^2 \| \widehat{M}  \|_{H^{-1/2}(\partial \Omega)}^2  }
    \right ),
\end{equation}
where $C>0$ is a positive constant only depending on $\alpha_0$, $\gamma_0$, $\Lambda_0$,
$M_0$, $M_1$, $\delta_0$, $F$. Assuming $\eta\leq r_0^2\| \widehat{M}  \|_{H^{-1/2}(\partial \Omega)}^2$, we obtain \eqref{eq:1.2}.

Without loss of generality, let $y_0\in \overline{\Omega\setminus D_1}$ such that
\begin{equation}
    \label{eq:7.1}
\hbox{dist}(y_0,\overline{\Omega\setminus D_2})=d.
\end{equation}

It is significant to assume $d>0$, so that $y_0\in D_2\setminus D_1$. Let us define
\begin{equation}
    \label{eq:7.2}
    h=\hbox{dist}(y_0,\partial D_1),
\end{equation}
possibly $h=0$.

There are three cases to consider:

i) $h\leq \frac{d}{2}$;

ii) $h> \frac{d}{2}$, $h\leq \frac{d_0}{2}$;

iii) $h> \frac{d}{2}$, $h> \frac{d_0}{2}$.

\medskip
\noindent Here the number $d_0$, $0<d_0<r_0$, is such that
$\frac{d_0}{r_0}$ only depends on $M_0$, and it is the same
constant appearing in Proposition \ref{eq:Cauchy1}. In particular,
Proposition $3.6$ in \cite{A-B-R-V} shows that there exists an
absolute constant $C>0$ such that if $d \leq d_0$, then $d\leq
Cd_m$.

\medskip
\noindent
\emph{Case i).}

By definition, there exists $z_0\in \partial D_1$ such that
$|z_0-y_0|=h$. By applying the triangle inequality, we get
$\hbox{dist}\left(z_0, \overline{\Omega\setminus D_2}\right)\geq
\frac{d}{2}$. Since, by definition, $\hbox{dist}\left(z_0,
\overline{\Omega\setminus D_2}\right)\leq d_m$, we obtain $d\leq
2d_m$.

\medskip
\noindent
\emph{Case ii).}

It turns out that $d<d_0$ and then, by the above recalled property, again we have that
$d\leq Cd_m$, for an absolute constant $C$.

\medskip
\noindent
\emph{Case iii).}

Let $\widetilde{h}=\min\{h,r_0\}$. We obviously have that $B_{\widetilde{h}}(y_0)\subset \Omega\setminus \overline{D_1}$ and
$B_d(y_0)\subset D_2$. Let us set

$$
d_1=\min\left\{\frac{d}{2},\frac{\overline{c_0}d_0}{4}\right\}.
$$
Since $d_1<d$ and $d_1<\widetilde{h}$, we have that
$B_{d_1}(y_0)\subset D_2\setminus \overline{D_1}$ and therefore
$\eta\geq \int_{B_{d_1}(y_0)}|\nabla^2 w_1|^2$.

Since $\frac{d_0}{2}<\widetilde{h}$,
$B_{\frac{d_0}{2}}(y_0)\subset \Omega\setminus \overline{D_1}$ so
that we can apply Proposition \ref{prop:SUCP_interior} with
$r_1=d_1$, $\overline{r}=\frac{d_0}{2}$, $r_2=\overline{c_0}\
\overline{r}$, obtaining $\eta\geq
C\left(\frac{2d_1}{d_0}\right)^{\tau_0}\int_{B_{\frac{d_0}{2}}(y_0)}|\nabla^2
w_1|^2$, with $C>0$ only depending on $\alpha_0$, $\gamma_0$,
$\Lambda_0$, $\delta_0$ and $F$. Next, by Proposition
\ref{prop:LPS}, recalling that $\frac{d_0}{r_0}$ only depends on
$M_0$, we derive that
$$
d_1
    \leq
    C r_0
    \left (
    \frac{ \eta   }{ r_0^2 \| \widehat{M}  \|_{H^{-1/2}(\partial \Omega)}^2  }
    \right )^{\frac{1}{\tau_0}},
$$
where $C>0$ only depends on $\alpha_0$, $\gamma_0$, $\Lambda_0$, $M_0$, $M_1$, $\delta_0$ and $F$. For $\eta$ small enough, $d_1<\frac{\overline{c_0}d_0}{4}$, so that
$d_1=\frac{d}{2}$ and
$$
d
    \leq
    C r_0
    \left (
    \frac{ \eta   }{ r_0^2 \| \widehat{M}  \|_{H^{-1/2}(\partial \Omega)}^2  }
    \right )^{\frac{1}{\tau_0}}.
$$
Collecting the three cases, we have

\begin{equation}
    \label{eq:9bis.1}
    d
    \leq
    C r_0
    \left (
    \frac{ \eta   }{ r_0^2 \| \widehat{M}  \|_{H^{-1/2}(\partial \Omega)}^2  }
    \right )^{\frac{1}{\tau_1}},
\end{equation}
with $\tau_1=\max\{\tau,\tau_0\}$ and $C>0$ only depends on $\alpha_0$, $\gamma_0$, $\Lambda_0$, $\alpha$, $M_0$, $M_1$, $\delta_0$ and $F$.

Finally, let us estimate $\delta$ in terms of $d$. By \eqref{eq:9bis.1}, for $\eta$ small enough, we have that
$$
d<\frac{r_0}{4\sqrt{1+M_0^2}}.
$$
Let us notice that if a point $y$ belongs to $\overline{D_1} \setminus D_2$, then $\hbox{dist}(y,\partial D_1)\leq d$.

Without loss of generality let $\overline{x}\in \overline{D_1}$
such that $\hbox{dist}(\overline{x},\overline{D_2})=\delta>0$.
Then $\overline{x} \in \overline{D_1} \setminus D_2$ and therefore
$\hbox{dist}(\overline{x},\partial D_1)\leq d$.

Let $w\in \partial D_1$ such that
$|w-\overline{x}|=\hbox{dist}(\overline{x},\partial D_1)\leq d$.

Letting $n$ the outer unit normal to $D_1$ at $w$, we may write
$\overline{x} =w-|w-\overline{x}|n$. By our regularity assumptions
on $D_1$, the truncated cone $C(\overline{x},
-n,2(\frac{\pi}{2}-\arctan M_0))\cap
B_{\frac{r_0}{4}}(\overline{x})$ having vertex $\overline{x}$,
axis $-n$ and width $2(\frac{\pi}{2}-\arctan M_0)$, in contained
in $D_1$.

On the other hand, by definition of $\delta$,
$B_\delta(\overline{x})\subset \Omega\setminus \overline{D_2}$, so
that the truncated cone $C(\overline{x},
-n,2(\frac{\pi}{2}-\arctan M_0))\cap
B_{\min\{\delta,r_0/4\}}(\overline{x})$ is contained in
$D_1\setminus \overline{D_2}$.

Let us see that $\delta<\frac{r_0}{4}$. In fact if, by
contradiction, $\delta\geq\frac{r_0}{4}$, we can consider the
point $z=\overline{x}-\frac{r_0}{4}n$. Since $z\in D_1\setminus
D_2$, as noticed above, $\hbox{dist}(z,\partial D_1)\leq d$. On
the other hand, by using the fact that $|z-w|\leq \frac{r_0}{2}$
and by the regularity of $D_1$, it is easy to compute that
$\hbox{dist}(z,\partial D_1)\geq \frac{r_0}{4\sqrt{1+M_0^2}}$,
obtaining a contradiction.

Hence $\min\{\delta, \frac{r_0}{4}\}=\delta$ and, by defining
$\overline{z}= \overline{x}-\delta n$ and by analogous
calculations, we can conclude that $\delta\leq (\sqrt{1+M_0^2})d$,
which is the desired estimate of $\delta$ in terms of $d$.

By Proposition \ref{prop:Cauchy1},
\begin{equation}
    \label{eq:ultima}
    d
    \leq
    C r_0
    \left (\log\left|\log\left(\frac{ \epsilon   }{ r_0^2 \| \widehat{M}  \|_{H^{-1/2}(\partial \Omega)}^2  }\right)\right|
    \right )^{-\frac{1}{2\tau_1}},
\end{equation}
with $\tau_1\geq 1$ and $C>0$ only depends on $\alpha_0$, $\gamma_0$,
$\Lambda_0$, $\alpha$, $M_0$, $M_1$, $\delta_0$ and $F$. By this
first rough estimate, there exists $\epsilon_0>0$, only depending
on on $\alpha_0$, $\gamma_0$, $\Lambda_0$, $\alpha$, $M_0$, $M_1$,
$\delta_0$ and $F$, such that, if $\epsilon\leq \epsilon_0$, then
$d\leq d_0$. Therefore the second part of Proposition
\ref{prop:Cauchy1} applies and the thesis follows.

\end{proof}

\section*{Acknowledgements}
Antonino Morassi is supported by PRIN 2015TTJN95 ``Identificazione
e diagnostica di sistemi strutturali complessi''. Edi Rosset and
Sergio Vessella are supported by Progetto GNAMPA 2017 ``Analisi di
problemi inversi: stabilit\`a e ricostruzione'', Istituto
Nazionale di Alta Matematica (INdAM). Edi Rosset is supported by
FRA2016 ``Problemi inversi, dalla stabilit\`a alla ricostruzione",
Universit\`a degli Studi di Trieste.

\end{document}